\documentclass{amsart}
\usepackage{amsfonts,latexsym,amssymb,amsmath,amsthm}
\usepackage{enumerate}

    \theoremstyle{plain}
    \newtheorem{theo}{Theorem}[section]
    \newtheorem{prop}{Proposition}[section]
    \newtheorem{coro}[theo]{Corollary}
    \newtheorem{lemma}{Lemma}[section]

    \theoremstyle{definition}
    \newtheorem{defin}{Definition}[section]
    \newtheorem{eje}{Example}[section]
    \newtheorem{nota}{Remark}[section]

\numberwithin{equation}{section}

\usepackage{hyperref}
\newcommand{\doi}[1]{\href{http://dx.doi.org/#1}{doi:#1}}
\newcommand{\arxiv}[1]{\href{http://arxiv.org/abs/#1}{arXiv:#1}}

\def\xx{{ x}}
\def\yy{{ y}}
\def\zz{{ z}}
\def\ww{{ w}}
\def\ee{{ e}}
\def\00{{ 0}}
\def\rr{{ r}}
\def\a{ \alpha}
\def\aa{ \alpha}
\def\bb{ \beta}

\def\AA{ A}
\def\N{\mathbb N}
\def\R{\mathbb R}
\def\C{\mathbb C}
\def\d{\partial}

\usepackage{amsfonts}
\usepackage{amsmath}
\usepackage{xcolor}
\usepackage{bm}

\begin{document}

\title[Formal $P$-Gevrey solutions for higher order PDEs]{Formal Gevrey solutions --in analytic germs-- for higher order holomorphic PDEs}

\author{Sergio A. Carrillo}

\address{Escuela de Ciencias Exactas e Ingenier\'{i}a, Universidad Sergio Arboleda, Calle 74, $\#$ 14-14, Bogot\'{a}, Colombia.}
\email{sergio.carrillo@usa.edu.co}

\author{Alberto Lastra}

\address{Universidad de Alcal\'{a}, Departamento de F\'{i}sica y Matem\'{a}ticas, Ap. de Correos 20, E-28871 Alcal\'{a} de Henares (Madrid), Spain.}
\email{alberto.lastra@uah.es}

	

\thanks{ Both authors are supported by the project PID2019-105621GB-I00 of Ministerio de Ciencia e Innovaci\'on, Spain. The first author is supported by the project ``An\'{a}lisis complejo, ecuaciones diferenciales y sumabilidad" (IN.BG.086.20.002 Univ. Sergio Arboleda). The second author is supported by Direcci\'on General de Investigaci\'on e Innovaci\'on, Consejer\'ia de Educaci\'on e Investigaci\'on of Comunidad de Madrid (Spain), and Universidad de Alcal\'a under grant CM/JIN/2019-010, Proyectos de I+D para J\'ovenes Investigadores, Univ. de Alcal\'a 2019.
}



\maketitle

\begin{abstract}
We consider a family of holomorphic PDEs whose singular locus is given by the zero set of an analytic map $P$ with $P(0)=0$. Our goal is to establish conditions for the existence and uniqueness of formal power series solutions and to determine their divergence rate. In fact, we prove that the solution is Gevrey in $P$, giving new information on divergency while compared to the classical Gevrey classes. If $P$ is not singular at $0$, we also provide Poincar\'e conditions to recover convergent solutions. Our strategy is to extend the dimension and lift the given PDE to a problem where results of singular PDEs can be applied. Finally, examples where the Gevrey class in $P$ is optimal are included. 
\end{abstract}

\section{Introduction}\label{Sec. Introduction}

The growth associated to the coefficients of formal solutions to functional equations has been widely studied in the literature. Results on this direction are known as \textit{Maillet type theorems}. They coined their name in honor to the  pioneering work of E. Maillet~\cite{maillet} (1903) where it was shown that any formal power series solution of a nonlinear algebraic ordinary differential equation is $s$-Gevrey, for some $s\geq 0$, see Section \ref{Sec. Gevrey series} for definitions. Further initial results in this context can be found in~\cite{malgrange,ramis,sibuya} where optimal bounds are interpreted as slopes of adequate Newton polygons associated to the given analytic equation. Recognizing optimal values for the Gevrey class of formal solutions is of utmost importance in the study of (Borel-, multi-)summability phenomena, a great tool to construct analytic solutions of the given problem which are asymptotic to the formal ones. 

The increasing interest on these results has provided  advances in other frameworks. For instance, on generalized power series solutions of ordinary differential equations~\cite{gontsovgoryuchkina}, in singularly perturbed problems~\cite{CDRSS}, integro-differential equations~\cite{remy}, moment PDEs~\cite{balseryoshino,lamisu,su}, difference and $q$-difference equations~\cite{immink,vizio,zhang}, among others. We can also mention  results in dynamical systems, such as the Gevrey character of invariant formal curves to analytic local diffeomorphisms~\cite{baldoma,lopez}.

Convergence and divergence (Maillet type) theorems have also been developed for singular holomorphic partial differential equations (of non-Kowalevski type, Fuchsian, of totally non-characteristic type, among others). A good account on these results can be found at Gerard and Tahara's book ~\cite{geta} and the references therein. Moreover, optimal Gevrey bounds have been found for many families of PDEs in terms of slopes of adequate Newton polygons, see, e.g., \cite{Hibino04,Hibino99,shirai,shirai2,yamazawa} and the recent work \cite{lastratahara}. The topic is an active subject of research where many problems on summability of solutions remain open.  

On the other hand, results on singular PDES are not directly applicable to other type of equations, for instance, mixing irregular singularities and singular perturbations. An interesting example is the family of doubly singular equations \begin{equation}\label{Eq. Intro Doubly}
\epsilon^\sigma z^{r+1}\frac{\d y}{\d z}=f(z,\epsilon,y),
\end{equation} where $\sigma$ and $r$ are positive integers and $f$ is analytic at the origin. The equation exhibits an irregular singularity at $z=0$ and a singular behavior as $\epsilon\to 0$. In this case, the optimal Gevrey type is only revealed when the equation is considered in the variable $t=z^r\epsilon^\sigma$. In fact, the relation between true solutions asymptotic to formal ones was answered in \cite{CDMS07} with the development of \textit{monomial summablity}. Later on, the extension of this notion to more variables led naturally to the study of equations of type \begin{equation}\label{Eq. Intro X}
\epsilon^\sigma x_1^{\a_1}\cdots x_n^{\a_n}\left( \lambda_1 x_1\frac{\d y}{\d x_1}+\cdots+\lambda_n x_n\frac{\d y}{\d x_n}\right)=f(x,\epsilon,y),
\end{equation} where $\lambda_1,\dots,\lambda_n>0$. This system is the higher dimension analogue to equation (\ref{Eq. Intro Doubly}). In this case, the optimal Gevrey type is obtained working with the variable $t=\epsilon^\sigma x_1^{\a_1}\cdots x_n^{\a_n}$. Moreover, novel results on the monomial summability of formal solutions are available in this framework \cite{Carr1}, see also \cite{YamazawaYoshino15} for the case $\a_j=0$.

Recently, the foundations of asymptotic expansions and summability with respect to an arbitrary analytic germ $P:(\C^d,0)\to 0$ such that $P(0)=0$ were established in \cite{mosc}. In particular, \textit{$P$-$k$-Gevrey series} were defined and systematized. Roughly speaking, a formal power series $\widehat{y}\in \C[[x]]$, $x=(x_1,\dots,x_d)$ is $P$-$k$-Gevrey if it can be written as \begin{equation}\label{Eq. y Intro}
	\widehat{y}=\sum_{n=0}^\infty y_n P^n,\quad \text{ where } \sup_{\xx\in D}|y_n(\xx)|\leq CA^n n!^{k},
\end{equation} for some constants $C,A>0$, and where the coefficients $y_n$ are holomorphic in a common polydisc $D\subseteq \C^d$ centered at the origin. This concept captures the idea of measuring the divergence of a series using the leading variable $t=P(x)$. Moreover, it gives more precise information on the divergence rate of $\widehat{y}$, inaccessible when only working with $x_1,\dots,x_d$ separately. 

In this setting, we can pose in greater generality the family of problems \begin{equation}\label{Eq. Intrp k=1}
P(x)L_1(y)=F(x,y),\qquad L_1:=a_1(x)\d_{x_1}+\cdots a_d(x)\d_{x_d},
\end{equation} with analytic coefficients, which include equations (\ref{Eq. Intro Doubly}) and (\ref{Eq. Intro X}) as particular cases. The key point to obtain existence and uniqueness of formal solutions of (\ref{Eq. Intrp k=1}) is that $$P \text{ divides } L_1(P).$$ Geometrically, this condition means that the local hypersurface $Z_P:=\{x\in(\C^d,0) : P(x)=0\}$ is invariant under the vector field $L_1$. In this case, the solution turns out to be $P$-$1$-Gevrey, as it was proved in \cite[Theorem 1]{cahu}.  Surprisingly, this recovered many cases on the Gevrey class of formal power series solutions of ODEs and PDEs that have been treated in the literature. Finally, results of this sort are a first step to approach Borel $P$-summability which is a difficult phenomenon far from being understood, see \cite{mosc,CMS19}.

The aim of this paper is to study a higher order analogue to (\ref{Eq. Intrp k=1}), where once again, known results in the theory of singular PDEs fail to provide optimal bounds for the Gevrey type of formal solutions. For positive integers $d,N,k$, and complex coordinates  $\xx=(x_1,\dots,x_d)\in(\C^d,\00)$ and $\yy=(y_1,\dots,y_N)\in\C^N$, 
we pose the system of PDEs \begin{equation}\label{e1general}
	P(\xx)^{k}L_k(\yy)(\xx)+\cdots+P(\xx)L_1(\yy)(\xx)=F(\xx,\yy).
\end{equation}
$F$ is a $\C^N$-valued holomorphic map defined near $(\00,\00)\in\C^d\times\C^{N}$, and  
\begin{equation}\label{L1general}
	L_j:=\sum_{|\aa|=j} a_{\aa}^{(j)} (\xx) \d_{\aa},\qquad j=1,\dots,k,
\end{equation} are differential operators of order $j$ with holomorphic coefficients $a_{\aa}^{(j)}$ near $0\in\C^d$, see below for notations. 
Note that if $x$ approaches $Z_P$, the nature of (\ref{e1general}) changes from differential to implicit one. Moreover, if the linear part of $F$ at the origin $D_yF(0,0)$ is an invertible matrix, $P$ cannot be canceled from (\ref{e1general}), so its zero set is a  non-removable singular part of the equation. We mention that this equation is also inspired in its simple one-dimensional analogue $$\tau^k b_k(\tau)\d_\tau^k(u)+\cdots+\tau b_1(\tau)\d_\tau(u)=f(\tau,u),$$ familiar from point of view of Borel summability.

The previous work \cite{cahu} studied equation (\ref{Eq. Intrp k=1}) by direct recurrences, based on generalized Weierstrass division algorithms, and used modified Nagumo norms \cite{CDRSS} to establish the Gevrey type in $P$ of $\widehat{y}$. However, this approach left several questions opened. First, do formal solutions of these equations admit a canonical expansion in power series of $P$? Second, is it possible to treat the families (\ref{Eq. Intrp k=1}) with the standard methods for nonlinear singular PDEs and Newton polygons? Here we answer both questions affirmatively for the more general equation (\ref{e1general}). The method we explore here consists of adding a time variable $t\in(\C,0)$ to lift  (\ref{e1general}) to a system of PDEs in $t$ and $x$. The new system will have a unique solution of the form $\widehat{W}(t,x)=\sum_{n=0}^\infty y_n t^n$, where the $y_n$ are as in (\ref{Eq. y Intro}). This trick produces an equation were known results on singular PDEs can be effectively used to find the Gevrey order in $t$ of $\widehat{W}$, and thus the $P$-Gevrey order of $\widehat{y}(x)=\widehat{W}(P(x),x)$. Since the lifted equation determines the coefficients $y_n$ naturally, this procedure guarantees a canonically decomposition of $\widehat{y}$ as a power series in $P$.  The idea was suggested in \cite{cahu} by anonymous referees to whom we thank for their contribution. 

To state our results, we associate to $L_j$ and $P$ the holomorphic function  
\begin{equation}\label{e157}
L_j^{\star}(P):=\sum_{|\aa|=j} a_{\aa}^{(j)}(\xx) (\d_{x_1}P)^{\aa_1}\dots (\d_{x_d}P)^{\aa_d},\qquad j=1,\dots,k.
\end{equation} In particular,  $L^{\star}_1(P)$ is simply $L_1(P)$, but for $j\geq2$ these expressions generally differ. It turns out that these functions contain the key that leads to the existence, uniqueness, and Gevrey order for formal solutions of (\ref{e1general}).

\begin{theo}\label{teopral}
Consider the system of partial differential equations (\ref{e1general}) where $F(\00,\00)=\00$, and $D_yF(0,0)\in\textup{GL}_N(\C)$ is an invertible matrix. If $L_k\not\equiv 0$ and 
\begin{equation}\label{e138}
P\hbox{ divides }L^{\star}_j(P),\hbox{ for every } j=1,\dots,k,
\end{equation}
then equation (\ref{e1general}) admits a unique formal power series solution $\widehat{\yy}\in\C[[\xx]]^{N}$ with $\widehat{y}(0)=0$. Moreover, $\widehat{\yy}$ is a $P$-$k$-Gevrey series.
\end{theo}

On the other hand, if $0$ is not a singular point for $P$, i.e., $\d_{x_l}P(0)\neq0$ for some $l$, theproblem changes and Poincar\'{e} type conditions appear to guarantee existence and uniqueness of solutions. In fact, we obtain an analytic solution.

\begin{theo}\label{Thm. Convergent} Consider  (\ref{e1general}) where $F(\00,\00)=\00$ and $D_yF(0,0)\in\textup{GL}_N(\C)$. If $L_k^\star(P)(0)\neq 0$ and

\begin{equation}\label{Eq. Condition Thm2}
\left[\sum_{j=1}^k  \frac{n!}{(n-j)!}L_j^\star(P)(0)\right] I_N - D_yF(0,0)\in\textup{GL}_N(\C),\, \text{ for all } n\geq 0,
\end{equation} then (\ref{e1general}) has a unique analytic solution $\widehat{y}\in \C\{x\}^N$ with $\widehat{y}(0)=0$. Here $I_N\in\C^{N\times N}$ is the identity matrix.
\end{theo}

We stress that the current technique can be applied to concrete equations and it is an idea worth exploring for future works. For instance, problems involving non-linear terms in the derivatives of $u$. In fact, obtaining $P$-Gevrey estimates for solutions of these problems is likely to be inaccessible by a direct approach.

The plan for the paper is as follows. Section \ref{Sec. Gevrey series} recalls the basics on Gevrey series in several variables and  $P$-Gevrey series, including a natural relation between them (Proposition \ref{Prop. Ps sss Gevrey}). The necessary tools to prove Theorems \ref{teopral} and \ref{Thm. Convergent} are developed in Sections \ref{Sec. A preliminary Maillet-type theorem} and \ref{Sec. Some technical results}. First, we give a Maillet type theorem for singular PDEs adapted for our purposes (Theorem \ref{Thm. Gerard-Tahara1}), and then several lemmas of elementary nature. The main results are proved in Section \ref{Sec. The proof of the main results}. 
The case $k=1$ is particularly simple and we include it in Corollaries \ref{teo2} and \ref{Coro 2} hoping that its proof helps to elucidate the ideas. The work concludes in Section \ref{Sec. Examples} with examples where the Gevrey type given by Theorem \ref{teopral} is attained.

\

\textbf{Notation}. $\N$ denotes the set of non-negative integers and $\N^\ast:=\N\setminus\{0\}$. For $d\in\N^\ast$,    $\aa=(\alpha_1,\ldots,\alpha_d), \bb=(\beta_1,\ldots,\beta_d)\in\N^d$, and $s=(s_1,\dots,s_d)\in\R_{\geq 0}$ we set  $$\aa+\bb=(\alpha_1+\beta_1,\ldots,\alpha_d+\beta_d),\quad   |\aa|=\alpha_1+\cdots+\alpha_d,\quad \aa!^s=\alpha_1!^{s_1}\dots\alpha_d!^{s_d}.$$ We write $\aa\le\bb$ if $\alpha_j\le \beta_j$, for all $1\le j\le d$, and $\aa<\bb$ if $\aa\le\bb$ and there is $1\le j_0\le d$ such that $\alpha_{j_0}<\beta_{j_0}$. If $\beta\leq \a$, we put  $\binom{\a}{\beta}=\binom{\a_1}{\beta_1}\cdots\binom{\a_d}{\beta_d}$. The symbol $\00$ stands for a vector with zero components. For $1\le j\le d$,  $\ee_j\in\N^d$ is the tuple with all its components being zero, except the position $j$ which is $1$. 

We work in $(\C^d,\00)$ with local coordinates $\xx=(x_1,\dots,x_d)$. If $\aa\in\N^{d}$, let $$\xx^{\aa}=x_1^{\alpha_1}\cdots x_d^{\alpha_d},\qquad \partial_{x_j}:=\partial_{\ee_j,\xx},\text{ and }\quad  \partial_{\aa,x}=\partial_{\aa}=\frac{\partial^{|\aa|}}{\partial x_1^{\a_1}\cdots \partial x_d^{\a_d}}.$$ In the former case we omit the $x$ when the variables are identified from the context. Given a complex Banach space $({E},\|\cdot\|)$, we write ${E}[[\xx]]$ and ${E}\{\xx\}$ for the spaces of formal and convergent power series in $\xx$ with coefficients in $E$, respectively. In our context, $E$ will be $\C^N$ or an adequate space of functions. If ${E}=\C$ we simply write $\widehat{\mathcal{O}}=\C[[\xx]]$ and  $\mathcal{O}=\C\{\xx\}$. $\mathcal{O}^\ast=\{U\in\mathcal{O} \,:\, U(\00)\neq 0\}$ is the group of units. 

Given $\hat{f}=\sum a_\bb \xx^\bb\in\widehat{\mathcal{O}}$, $o(\hat{f})$ denotes its order: if $\hat{f}=\sum_{n=0}^\infty {f}_n$, ${f}_n=\sum_{|\bb|=n} a_\bb \xx^\bb$, is written as sum of its homogeneous components, $o(\hat{f})$ is the least integer $k$ for which ${f}_k\neq 0$. Given a polyradius $R=(R_1,\ldots,R_d)\in \R_{>0}^{d}$, we write $$D_{R}:=\{x\in\mathbb{C}^{d}: |x_j|<R_j, j=1,\ldots d\},$$ for such polydisc. If $R=(r,\dots,r) $, $r>0$, we also write $D_R=D_r^d$ as the Cartesian product of one-dimensional discs. For $N\in\N^\ast$ we set  $\mathcal{O}(\Omega,\mathbb{C}^{N})$ (resp. $\mathcal{O}_b(\Omega,\mathbb{C}^{N})$) for the set of $\mathbb{C}^{N}$-valued holomorphic (resp. and bounded) functions  on an open domain $\Omega\subseteq\mathbb{C}^{d}$. We  write  $\mathcal{O}(\Omega):=\mathcal{O}(\Omega,\mathbb{C})$ and $\mathcal{O}_b(\Omega):=\mathcal{O}_b(\Omega,\mathbb{C})$ for short. Note that $\mathcal{O}_b(\Omega,\mathbb{C}^{N})$ endowed with the supremum norm is a Banach space.

\section{Gevrey series}\label{Sec. Gevrey series}

We start by recalling the main facts on Gevrey series in several variables and with respect to germs of analytic functions. In particular, we include a relation between these notions which was first obtained in the proceeding article \cite{cahu}.

\begin{defin}\label{def132} Let $E$ be a  complex Banach space and $s=(s_1,\dots,s_d)\in\R_{\geq 0}^d$. A series $\hat{f}=\sum_{\bb\in\N^d} a_\bb \xx^\bb\in E[[\xx]]$ is \textit{$s$-Gevrey} if we can find $C,A>0$ such that $$\|a_\bb\|\leq CA^{|\bb|} \bb!^{s},\quad \text{ for all } \bb\in\N^d.$$ Equivalently, $\sum_{\bb\in\N^d} a_\bb \xx^\bb/\bb!^s\in E\{x\}$. Note that $s=\00$ means convergence. In the case $p=s_1=\dots=s_d\geq 0$, since $\bb!\leq |\bb|!\leq d^{|\bb|}\bb!$, $\hat{f}$ is $(p,\dots,p)$-Gevrey if and only if there are $C,A>0$ such that \[\|a_\bb\|\leq CA^{|\bb|} |\bb|!^p,\quad \bb\in\N^d.\] 
\end{defin}

We denote by $E[[x]]_{s}$ the set of $s$-Gevrey series with coefficients in $E$. This space is closed under sums and partial derivatives, and it contains $E\{x\}$. It is also closed under products when $E$ is a Banach algebra. Moreover, it is stable under linear changes of variables, see \cite[Lemma 2.1]{Hibino99}.

\begin{lemma}\label{Lemma linear change} Given $p\geq 0$, $\hat{f}(\xx)\in E[[\xx]]_{(p,\dots,p)}$ if and only if $\hat{f}(M\xx)\in E[[\xx]]_{(p,\dots,p)}$, for all $M\in\textup{GL}_d(\C)$.
\end{lemma}

Consider now a germ $P$ at $\00\in\C^d$ of a $\C$-valued holomorphic function, i.e., an element $P\in\mathcal{O}\setminus\{0\}$, and assume that $P(\00)=0$. There are equivalent definitions for Gevrey series with respect to $P$, with coefficients in $E$, see \cite{CMS19,mosc}. We focus on the case $E=\C$ and follow the simple characterization given in \cite[Lemma 4.1]{CMS19}. 

\begin{defin}\label{Def.Gevrey} Given $s\geq 0$, $\hat{f}\in\widehat{\mathcal{O}}$ is said to be a \textit{$P$-$s$-Gevrey series} if there is a polyradius $r$, constants $C,A>0$ and a sequence $\{f_n\}_{n\in\N}\in\mathcal{O}_b(D_\rr)$ such that \begin{equation}\label{Def. P s Gevrey}
	\hat{f}=\sum_{n=0}^\infty f_n P^n,\quad \text{ where }  \sup_{\xx\in D_r}|f_n(\xx)|\leq CA^n n!^s.
	\end{equation} We will use the notation $\widehat{\mathcal{O}}^{P,s}$ for the set of $P$-$s$-Gevrey series. A series  $(\hat{f}_1,\dots,\hat{f}_N)\in\widehat{\mathcal{O}}^N$ is $P$-$s$-Gevrey if every component is so.
\end{defin}

\begin{nota} The expansion (\ref{Def. P s Gevrey}) is not unique. In fact, for each injective linear form $\ell:\N^d\to\R$ there is one such decomposition via a generalized Weierstrass division theorem, see \cite{mosc,cahu}. In general, the $f_n$ obtained from $\hat{f}$ under this process are merely formal power series. Therefore, in our definition we are implicitly assuming that these coefficients are convergent in a common polydisc at $\00\in\C^d$. Moreover, the growth of $f_n$ does not dependent on the decomposition used, thus the notion of $P$-$s$-Gevrey series is well-defined, see \cite[Lemma 4.1]{CMS19} for details.
\end{nota}

The following properties are valid for $s\geq0$ and $P,Q\in\mathcal{O}\setminus\{0\}$ such that $P(\00)=Q(\00)=0$, c.f., \cite[Corollary 4.2, Lemma 4.3]{CMS19}:

\begin{enumerate}
	\item $\widehat{\mathcal{O}}^{P,s}$ is stable under sums, products and partial derivatives, and $\mathcal{O}\subset \widehat{\mathcal{O}}^{P,s}$.
	
	\item For any $k\in\N^\ast$, $\widehat{\mathcal{O}}^{P^k,ks}=\widehat{\mathcal{O}}^{P,s}$.
	
	\item If $Q$ divides $P$, then $\widehat{\mathcal{O}}^{P,s}\subseteq \widehat{\mathcal{O}}^{Q,s}$. In particular, if $Q=U\cdot P$, $U\in\mathcal{O}^\ast$, then $\widehat{\mathcal{O}}^{P,s}=\widehat{\mathcal{O}}^{Q,s}$. 
	
	\item Let $\phi:(\C^d,\00)\to(\C^d,\00)$ be analytic, $\phi(\00)=\00$, and assume $P\circ \phi$ is not identically zero. If $\hat{f}\in \widehat{\mathcal{O}}^{P,s}$, then $\hat{f}\circ \phi\in \widehat{\mathcal{O}}^{P\circ \phi,s}$.
	
	\item If $P(\xx)=\xx^\aa$, $\aa\in\N^d\setminus\{\00\}$, then $\hat{f}=\sum a_{{\beta}}{x}^{{\beta}}\in \widehat{\mathcal{O}}^{\xx^\aa,s}$ if and only if there are constants $C,A>0$ satisfying \begin{equation}\label{Eq. Gevrey bounds monomials}
	|a_{{\beta}}|\leq CA^{|{\beta}|}\min\{ \beta_j!^{s/\a_j} : j=1,\dots,d, \a_j\neq 0\} ,\quad {\beta}\in\N^d.
	\end{equation}  Note that the variables $x_j$ for which $\a_j=0$ can be regarded as regular parameters.
\end{enumerate}

The previous statement characterizes $P$-$s$-Gevrey series when $P$ is a monomial, directly from the growth of the coefficients of the series. Although it is not yet known whether a similar property is true for an arbitrary $P$, we have the following result from \cite[Proposition 3]{cahu} that we include for the sake of completeness.

\begin{prop}\label{Prop. Ps sss Gevrey} Consider $P\in \mathcal{O}$ with $o(P)=k\geq 1$. Then, a $P$-$s$-Gevrey series is a $(s/k,\dots,s/k)$-Gevrey series. 
\end{prop}

\begin{proof} Writing $P=\sum_{j=k}^\infty P_j$ as sum of homogeneous polynomials, where $P_k\neq 0$, take ${a}\in\C^d$ such that $P_k({a})\neq 0$, and choose $A\in\text{GL}_n(\C)$ having ${a}$ as first column. If we set $Q(\xx)=P(A\xx)$ and we write it as sum of its homogeneous components $Q=\sum Q_j$, then $Q_j(\xx)=P_j(A\xx)$, and $Q_k(\xx)=P_k({a})x_1^k+\cdots$, i.e., $o(Q)=k$ and $Q_k(1,0,\dots,0)\neq 0$.

Given a $P$-$s$-Gevrey series $\hat{f}$, the series $\hat{f}_0(\xx)=\hat{f}(A\xx)=\sum b_\bb \xx^\bb$ is a $Q$-$s$-Gevrey series, thanks to (4) above. We consider the change of variables \begin{equation}\label{Eq. Blow up}
x_1=z_1,\quad x_2=z_1z_2,\quad \dots,\quad x_d=z_1z_d.
\end{equation}  If $R(\zz)=Q(\xx)$ and $\hat{f}_1(\zz)=\hat{f}_0(\xx)$, we see that $\hat{f}_1$ is a $R$-$s$-Gevrey series. Now, \[R(\zz)=Q(z_1,z_1z_2,\dots,z_1z_d)=\sum_{j=k}^\infty z_1^j Q_j(1,z_2,\dots,z_d)=z_1^k U(\zz),\] where $U$ is a unit, because $U(\00)=Q_k(1,0,\dots,0)\neq 0$. Using this equation and (2) above, we find that $\hat{f}_1$ is $z_1^k$-$s$-Gevrey, or equivalently, a $z_1$-$s/k$-Gevrey series. Let us write $z'=(z_2,\dots,z_d)$. Since  \[
\hat{f}_1(\zz)=\sum_{\bb\in\N^d} b_\bb z_1^{|\bb|} z_2^{\beta_2}\cdots z_d^{\beta_d}\\
=\sum_{{(n,{\gamma})\in\N\times\N^{d-1}}\atop {n\geq |{\gamma}|}} b_{n-|{\gamma}|,{\gamma}} z_1^n \zz'^{{\gamma}},\] we can find constants $C,A>0$ such that $|b_{n-|{\gamma}|,{\gamma}}|\leq CA^{n+|{\gamma}|}n!^{s/k}$. Therefore, \[|b_\bb|\leq CA^{\beta_1+2\beta_2+\cdots+2\beta_d} |\bb|!^{s/k},\quad \text{ } \bb\in\N^d,\] i.e.,  $\hat{f}_0$ is $(s/k,\dots,s/k)$-Gevrey. The same is true for $\hat{f}$ due to Lemma \ref{Lemma linear change}.
\end{proof}

\section{A preliminary Maillet-type theorem for singular PDEs}\label{Sec. A preliminary Maillet-type theorem}

The aim of this section is to establish the existence, uniqueness, and Gevrey class (in the time variable $t$) of formal solutions of a family of singular PDEs. These include the equations that will be obtained by lifting (\ref{e1general}). The results presented here will be the key to prove Theorems \ref{teopral} and \ref{Thm. Convergent}.

More precisely, fixing $m,d,N\in\N^\ast$, $p,k\in\N$ and local coordinates $(t,x)\in(\C\times\C^d,0)$, we consider the system of equations
\begin{equation}\label{Eq. Main GT} 
	[c_p(\xx)(t\d_t)^p+\cdots+c_1(\xx)(t\d_t)+c_0(\xx)]u=B(\xx)t^k+G(x)(t,D^mu).
\end{equation} for an unknown $u=u(t,x)\in\C^N$. The coefficients in (\ref{Eq. Main GT}) are assumed to be holomorphic and bounded near the origin, say $c_0,\dots,c_p\in\mathcal{O}_b(D_r^d,\C^{N\times N})$ and $B\in \mathcal{O}_b(D_r^d,\C^N)$ for a fixed $r>0$. Moreover, $G(\xx)(t,D^mu)$ is the operator 
$$u(t,\xx)\mapsto G(\xx)(t,D^mu):=G_0(t,\xx,u)+\sum_{(b,\aa)\in I_m} G_{b,\aa}(t,\xx) t^b\d_t^b\d_{\aa,\xx} u,$$ acting on $\mathbb{C}[[t,\xx]]^N$, where:

\begin{itemize}
\item $I_m:=\{(b,\aa)\in \N\times \N^d  \,:\, b+|\aa|\leq m\}$ is a finite set of indices.

\item $G_0\in\mathcal{O}_b(D_r\times D_r^d\times D_r^N,\C^N)$ and $G_{b,\a}\in\mathcal{O}_b(D_r\times D_r^d)$, for all $(b,\aa)\in I_m$. 

\item The previous maps have the convergent Taylor expansions
\begin{align*}
G_0(t,x,u)=\sum_{j=0}^\infty F_{0,j}(x,u) t^j,\quad\text{ and }\quad  G_{b,\aa}(t,\xx)=\sum_{j=1}^\infty  g_{b,\aa,j}(\xx)t^j,
\end{align*} respectively. We assume that  $$F_{0,j}(x,u)=\sum_{\gamma\in\N^N,  |{\gamma}|\geq 2} F_{0,j,{\gamma}}(\xx) u^{{\gamma}},$$ has only non-linear terms in $u$, where $F_{0,j,\gamma}\in\mathcal{O}_b(D_r^d,\C^N)$ and $g_{b,\a,j}\in\mathcal{O}_b(D_r^d)$. Thus, the non-linear terms in $u$ of $G$ are collected in $G_0$ whereas the remaining terms are linear in $u$ and its derivatives.
\end{itemize}

Equation (\ref{Eq. Main GT}) is part of a family of scalar equations ($N=1$) treated in \cite[Chapter 6]{geta} for $k=1$. In that case, the Gevrey class is given by the maximum of \begin{equation}\label{Eq. Gevrey class}
	s_p(t^{j+b}\d_t^b\d_{\aa}):=\max\left\{0,\frac{b+|\aa|-p}{j}\right\},
\end{equation} and taken over the terms appearing on the right-hand side of (\ref{Eq. Main GT}). Our adaptation below will be obtained from this statement which is Theorem 6.3.1 and Corollary 6.3.3 (1) for $p=0$ in \cite{geta} (and $d=1$ in their notation). 

\begin{theo}[Gerard-Tahara]\label{Thm. Gerard-Tahara1} A sufficient condition to guarantee the existence and uniqueness of a solution of equation (\ref{Eq. Main GT}) of the form \begin{equation}\label{Eq. uk}
		\widehat{u}(t,x)=\sum_{n=k}^\infty u_n(x) t^n\in \mathcal{O}_b(D_\rho^d,\C^N)[[t]],\quad \text{ for some } \rho>0,
	\end{equation} is that  \begin{equation}\label{Eq. cpc0 invertible}
c_p(0) \text{ and } c_p(0)n^p+\cdots+c_1(0)n+c_0(0)\text{ are invertible for all } n\geq 0.
\end{equation} In this case, $\widehat{u}$ is $s$-Gevrey in $t$, where \begin{equation}\label{Eq. s}
		s:=\sup_{(j,b,\aa)\in J} s_p(t^{j+b}\d_t^b\d_{\aa}),\end{equation} and $J=\{ (j,b,\aa)\in\N^\ast \times\N\times\N^d  : g_{b,\aa,j}(\xx)\not\equiv 0\}$.
\end{theo}

\begin{proof} If we substitute $\widehat{u}(t,x)=\sum_{n=0}^\infty u_n(x) t^n$ into (\ref{Eq. Main GT}) and equate coefficients in corresponding powers of $t$, we find that \begin{equation}\label{Eq. aux implicit GT}
c_0(x)u_0(x)=F_{0,0}(x,u_0(x)).
\end{equation} Moreover, for $n\geq 1$ we have the recurrence \begin{align}
\nonumber  [c_p(x)n^p+&\cdots+c_1(x)n+c_0(x)]u_{n}(x)\\
\label{Eq. Aux GT} &=\delta_{n,k}B(x)+\sum_{l=1}^{n-1} \sum_{(b,\a)\in I_m} \binom{l}{b} b! g_{b,\aa,n-l}(x) \d_\aa(u_l) +\text{ l.o.t},
\end{align} where l.o.t. are the non-linear terms in $u_1,\dots,u_{n-1}$ coming from $G_0(t,x,\widehat{u}(t,x))$, and $\delta_{n,k}$ is the Kronecker delta. 

Condition (\ref{Eq. cpc0 invertible}) allows to determine uniquely the coefficients $u_{n}(x)$, $n\geq 1$, from (\ref{Eq. Aux GT}) thanks to the following lemma. To not interrupt the discussion, we postpone the proof to the end of the section.

\begin{lemma}\label{Lema matrices} Consider $c_0,\dots,c_p\in\mathcal{O}_b(D_r^d,\C^{N\times N})$ such that (\ref{Eq. cpc0 invertible}) holds. Then there is $0<\rho\leq r$ such that  $c_p(x)n^p+\cdots+c_1(x)n+c_0(x)$ is invertible, for all $x\in D_\rho^d$ and $n\geq 0$. Moreover, there is a constant $M>0$ such that \begin{equation}\label{Eq. Bound inverses}
		\sup_{x\in D_\rho^d} \left\|(c_p(x)n^p+\cdots+c_1(x)n+c_0(x))^{-1}\right\|\leq \frac{M}{n^p},\qquad \text{ for all } n\geq 1.
	\end{equation} Here $\|B\|=\max_{1\leq i\leq N} \sum_{j=1}^N |B_{i,j}|$, for $B=(B_{ij})\in\C^{N\times N}$.
\end{lemma}

We have seen that $u_n\in \mathcal{O}_b(D_\rho^d,\C^N)$ can be found recursively from $u_0$. Now, to determine $u_0(x)$ we apply the implicit function theorem which shows that (\ref{Eq. aux implicit GT}) has a unique analytic solution $u_0(x)\in\C\{x\}^N$ such that $u_0(0)=0$. Since $u_0=0$ also solves this equation, the initial term of $\widehat{u}$ is $u_0(x)\equiv 0$. Moreover, (\ref{Eq. Aux GT}) shows that $u_0=u_1=\cdots=u_{k-1}=0$ while $u_k(x)=(c_p(x)k^p+\cdots+c_1(x)k+c_0(x))^{-1} B(x)$. In this way, we see that the system (\ref{Eq. Main GT}) has a unique formal power series solution of the form (\ref{Eq. uk}).

We proceed with the Gevrey type. The result holds for $k=1$ since the majorant argument in \cite{geta} can be modified for vector equations in a straightforward way, see also Remark \ref{Rrk. Poincare condition} below. It is worth remarking that the reason the term $-p$ appears in (\ref{Eq. Gevrey class}) is due to the inequality (\ref{Eq. Bound inverses}) ---in \cite[p.180]{geta} it is used in the equivalent form $\|L_nu_n\|_r\geq (\sigma_0/2)^p n^p \|u_n\|_r$---.

The case $k>1$ is done using the change of variables $$u(t,x)=t^{k-1}v(t,x).$$ We can check that $\widehat{u}$ in (\ref{Eq. uk}) solves (\ref{Eq. Main GT}) if and only if $\widehat{v}=t^{-(k-1)}\widehat{u}$ solves an equation of the same type but with $k=1$. In fact, a direct calculation using Leibniz rule to compute $(t\d_t)^b(t^{k-1}v)$ and $t^b\d_t^b(t^{k-1}v)$ shows that $u$ satisfies (\ref{Eq. Main GT}) if and only if $v$ satisfies $$[\widetilde{c}_p(\xx)(t\d_t)^p+\cdots+\widetilde{c}_1(\xx)(t\d_t)+\widetilde{c}_0(\xx)]v=B(\xx)t+\widetilde{G}(x)(t,D^mv).$$ The new coefficients are $\widetilde{c}_l=\sum_{j=l}^p \binom{j}{l}(k-1)^{j-l}c_j$, $l=0,1,\dots,p$,
$$\widetilde{G}(x)(t,D^mv)=\widetilde{G}_0(t,x,v)+\!\sum_{(b,\aa)\in I_m} \sum_{l=0}^{b} \binom{b}{l} 
\binom{k-1}{b-l}(b-l)! G_{b,\aa}(t,\xx) t^{l} \d_t^{l}\d_\aa(v),$$ where $G_0(t,x,t^{k-1}v)=t^{k-1}\widetilde{G}_0(t,x,v)$ and $$\widetilde{G}_0(t,x,v)=\sum_{j=0}^\infty \widetilde{F}_{0,j}(t,x,v) t^j,\quad \widetilde{F}_{0,j,\gamma}(t,x,v):=\sum_{{\gamma\in\N^N}\atop {|{\gamma}|\geq 2}} F_{0,j,{\gamma}}(\xx) t^{(k-1)(|\gamma|-1)} v^{{\gamma}}.$$ They remain holomorphic and bounded near the origin. Moreover, the condition (\ref{Eq. cpc0 invertible}) holds in this case since $\widetilde{c}_p(x)=c_p(x)$ and 
$$\sum_{l=0}^p n^l\widetilde{c}_l(x) =\sum_{j=0}^p (k-1+n)^{j}c_j(x).$$ Thus these matrices are invertible at $x=0$, for all $n\geq 0$. By the case $k=1$, $\widehat{v}$ is of $s$-Gevrey, where $s$ is the maximum of $s_p(t^{j+l}\d_t^l \d_\a)$ over the indexed $(j,l,\a)$ such that $0\leq l\leq b$, $b+|\a|\leq m$, and $g_{b,\a,j}(x)\neq 0$. But (\ref{Eq. Gevrey class}) shows that $$\max_{0\leq l\leq b} s_p(t^{j+l}\d_t^l \d_\a)=s_p(t^{j+b}\d_t^b \d_\a).$$ Therefore, $s$ is given by (\ref{Eq. s}). Since multiplication by $t^{k-1}$ does not change the Gevrey order of a series, $\widehat{u}$ is also $s$-Gevrey as we wanted to prove.
\end{proof}

\begin{nota}\label{Rrk. Poincare condition}
The invertibility of  
$c_p(0)n^p+\cdots+c_1(0)n+c_0(0)$ means that $$C(\lambda):=\det(c_p(0)\lambda^p+\cdots+c_1(0)\lambda+c_0(0))\neq 0,\qquad\text{ for } \lambda=n\in\N.$$ Since $c_p(0)$ is also invertible, the function $C(\lambda)$ is a polynomial in $\lambda$ of degree exactly $Np$. If we denote its roots by $\lambda_1,\dots,\lambda_{Np}\in\C$, we are requiring that $\lambda_j\neq n$, for all possible $j$ and  $n$. This is equivalent to the existence of a constant $\sigma>0$ such that $$\left|n-\lambda_j\right|> \sigma n,\qquad \text{ for all } j=1,\dots, Np,\, n\in\N,$$ which is the classical \textit{Poincar\'{e} condition}, c.f., \cite[Theorem 6.3.1]{geta}.
\end{nota}

\begin{nota}\label{Rrm Stirling} An equivalent form of equation (\ref{Eq. Main GT}) is \begin{equation}\label{Eq. Main GT*}
[c'_p(\xx)t^p\d_t^p+\cdots+c'_1(\xx)t\d_t+c'_0(\xx)]u=B(\xx)t^k+G(x)(t,D^mu).
\end{equation} In this case, the hypothesis on the matrices is that $$c_p'(0)\text{ and } \sum_{j=0}^p n(n-1)\cdots (n-1+j)c_j'(0)\text{ are invertible for all } n\geq 0.$$ This can be checked recalling the \textit{Stirling numbers of the first kind} $s(j,l)\in\mathbb{Z}$, $1\leq l\leq j$, which are defined by the expansion $$\lambda(\lambda-1)\cdots (\lambda-1+j)=\sum_{l=1}^j s(j,l) \lambda^l,\,\, \text{ and satisfying }\,\, t^j\d_t^j=\sum_{l=1}^j s(j,l) (t\d_t)^l.$$ Writing the left-hand side of (\ref{Eq. Main GT*}) in terms of the operators $(t\d_t)^j$, it takes the form of (\ref{Eq. Main GT}) with 
$$c_p(x)=c_p'(x),\quad c_l(x)=\sum_{j=l}^p s(j,l) c_j'(x),\, l=0,1\dots,p-1.$$ Thus $\sum_{l=0}^p c_l(x)n^l=\sum_{j=0}^p n(n-1)\cdots (n-1+j)c_j'(x)$ as required.
\end{nota}

We conclude the section with the proof of the lemma.

\begin{proof}[Proof of Lemma \ref{Lema matrices}] Since $c_0(0), c_p(0)$ are invertible we can choose $\rho>0$ such that $c_0(x), c_p(x)$ are invertible and $c_0(x)^{-1}, c_p(x)^{-1}\in \mathcal{O}_b(D_\rho^d,\C^{N\times N})$. 
	
We recall that if $B=(B_{ij})\in\C^{N\times N}$ is such that $\|B\|<1$ for a matrix norm $\|\cdot \|$, then $I_N-B$ is invertible with inverse given by the Neumann series $(I_N-B)^{-1}=\sum_{n=0}^\infty B^n$. Moreover $\|(I_N-B)^{-1}\|\leq {1}/{(1-\|B\|)}.$ In particular, this holds for  $B\in\mathcal{O}_b(D_\rho^d,\C^{N\times N})$ and the supremum norm  $\|B\|_{\rho}:=\sup_{x\in D_\rho^d} \|B(x)\|$, where $\|\cdot\|$ is as in the statement of the lemma.
	
Consider an integer $n>L=\|c_p^{-1}\|_{\rho} \sum_{j=1}^{p-1}  \|c_j\|_{\rho}$. If $x\in D_\rho^d$, then  
	$$\left\|\left(\frac{c_0(x)}{n^p}+\frac{c_1(x)}{n^{p-1}}+\cdots+\frac{c_{p-1}(x)}{n}\right) c_p^{-1}(x) \right\|\leq \sum_{j=0}^{p-1} \frac{\|c_j\|_{\rho}}{n} \|c_p^{-1}\|_{\rho}<1.$$ By the previous paragraph, we find that $$c_p(x)n^p+\cdots+c_1(x)n+c_0(x)=\left(I_N+\left(\frac{c_0(x)}{n^p}+\cdots+\frac{c_{p-1}(x)}{n}\right) c_p^{-1}(x)\right) n^p c_p(x),$$ is invertible, for all $x\in D^d_\rho$. Moreover, we have the bound \begin{align*}
		&\| (c_p(x)n^p+\cdots+c_1(x)n+c_0(x))^{-1}\| \leq  \frac{\|c_p^{-1}\|_\rho/n^p}{1-\left\|\left(\frac{c_0(x)}{n^p}+\cdots+\frac{c_{p-1}(x)}{n}\right) c_p^{-1}(x)\right\|}\\
		& \leq \frac{\|c_p^{-1}\|_\rho/n^p}{1-\|c_p^{-1}\|_\rho\left(\frac{\|c_0\|_{\rho}}{n^p}+\cdots+\frac{\|c_{p-1}\|_{\rho}}{n}\right)}\leq \frac{1}{an^p-\left(\|c_0\|_{\rho}+\cdots+\|c_{p-1}\|_{\rho}n^{p-1}\right)},
	\end{align*} where $a=1/\|c_p^{-1}\|_\rho$. This shows that (\ref{Eq. Bound inverses}) holds for a large $M$. Note also that the denominator is indeed positive since, by hypothesis, $\|c_p^{-1}\|_{\rho} \sum_{j=1}^{p-1}  \|c_j\|_{\rho} n^j\leq n^{p-1} \sum_{j=1}^{p-1}  \|c_j\|_{\rho}/a<n^p$.  For the remaining integers $1\leq n\leq L$, by (\ref{Eq. cpc0 invertible}) we can shrink $\rho$ and enlarge $M$ if necessary to assure that  $c_p(x)n^p+\cdots+c_1(x)n+c_0(x)$ is invertible, for all $x\in D_\rho^d$ and that (\ref{Eq. Bound inverses}) still holds as it was required.
\end{proof}

\section{Some technical results}\label{Sec. Some technical results}

The proof of Theorems \ref{teopral} and \ref{Thm. Convergent} requires some technical lemmas that we collect here. They contain elementary properties on the derivatives of powers of a function and on suitable changes of variables. 

Although we are mainly interested in holomorphic coefficients, we state the following two results for arbitrary formal power series. We recall that according to the notation in (\ref{e157}) we have that \begin{equation}\label{e337}
	\d_{\aa}^\star(P):=(\partial_{x_1}P)^{\alpha_1}\cdots (\partial_{x_d}P)^{\alpha_d},\qquad \alpha=(\alpha_1,\ldots,\alpha_d)\in\N^d\setminus\{0\}.
\end{equation}

\begin{lemma}\label{lema178}
Consider $P\in \C[[\xx]]$, $\aa\in \N^{d}\setminus\{\00\}$ and an integer $n\geq 1$. Then,  
\begin{equation}\label{Eq. DaP^n}
\partial_{\aa}(P^n)=\sum_{j=1}^{n}\frac{n!}{(n-j)!}P^{n-j}\cdot\AA_{\aa,j},
\end{equation} where  each $\AA_{\aa,j}$ is a polynomial in derivatives of $P$, and it does not depend on $n$. In particular, $\AA_{\aa,1}=\d_\alpha(P)$, $A_{\aa,j}=0$ if $j>|\aa|$ and \begin{equation}\label{Eq. Aa|a|}
\AA_{\aa,|\alpha|}=\d_{\aa}^\star(P).    
\end{equation}
\end{lemma}

\begin{proof}
We apply induction  on $|\aa|$. The result is valid for $|\aa|=1$ and 
\begin{equation}\label{e186}
\AA_{\ee_l,1}:=\partial_{\ee_l}(P),\qquad l=1,\ldots,d,
 \end{equation} since $\partial_{\ee_l}(P^n)=nP^{n-1}\partial_{\ee_l}(P)$. If we assume the result is valid up to some $|\aa|$, the induction argument shows that 
\begin{align*}
\partial_{\aa+\ee_l}(P^n)&=\sum_{j=1}^{n} \frac{n!}{(n-j)!} \partial_{\ee_l}(P^{n-j} \AA_{\aa,j})\\
&=\sum_{j=1}^{n-1} \frac{n!}{(n-j-1)!}P^{n-j-1} \partial_{\ee_l}(P) \AA_{\aa,j}+\sum_{j=1}^n \frac{n!}{(n-j)!}P^{n-j} \partial_{\ee_l}( \AA_{\aa,j}),
\end{align*} for $l=1,\dots,d$. A rearrangement of the terms in the previous expression leads to (\ref{Eq. DaP^n}) for $\aa+e_l$ where
\begin{align}\label{e196}
\AA_{\aa+\ee_l,1}&=\partial_{\ee_l}(\AA_{\aa,1}),\\
\label{Eq. 21}
\AA_{\aa+\ee_l,j}&=\partial_{\ee_l}(\AA_{\aa,j})+\partial_{\ee_l}(P)\cdot \AA_{\aa,j-1},\qquad j=2,,\dots, n.
\end{align} Then (\ref{Eq. DaP^n}) holds for $|\aa|+1$. The formula follows from the principle of induction. 

On the other hand, it is clear from (\ref{e186}) and (\ref{e196}) that $\AA_{\aa,1}=\d_{\aa}(P)$ is valid. In addition to this, if $j>|\aa|$,  the recurrence (\ref{Eq. 21}) implies that $A_{\aa,j}=0$. Finally, if $j=|\aa|$, (\ref{Eq. 21}) takes the form $$\AA_{\ee_l,1}=\d_{\ee_l}(P),\qquad  \AA_{\aa+\ee_l,|\aa|+1}=\partial_{\ee_l}(P)\cdot \AA_{\aa,|\aa|},$$ from which (\ref{Eq. Aa|a|}) follows.
\end{proof}

\begin{nota} Equation (\ref{Eq. 21}) describes a recursion leading to each $\AA_{\aa,j}$. We can give closed formulas for them using the multivariate Fa\`{a} di Bruno formula \cite[p. 505]{cosa}. Indeed, consider  $h(\xx)=f(g^{(1)}(\xx),\ldots, g^{(n)}(\xx))$, where  
$$f(y_1,\ldots,y_n)=y_1\cdots y_n,\qquad g^{(1)}(\xx)\equiv\ldots\equiv g^{(n)}(\xx)\equiv P(\xx).$$

Notice that $\d_{{\lambda}}(f)(P,\dots,P)=P^{n-j}$ if ${\lambda}\in\{0,1\}^n$ and $|{\lambda}|=j$, and $\d_{{\lambda}}(f)=0$ otherwise. Then, for every $\aa\in\N^d\setminus\{0\}$ one has
\begin{equation}\label{e169}
\partial_{\aa}(P^n)=\sum_{j=1}^{|\aa|}  P^{n-j}\sum_{{{\lambda}\in\{0,1\}^{n}}\atop {|{\lambda}|=j}} \left[\sum_{s=1}^{|\aa|}\sum_{p_s(\aa,{\lambda})} \aa! \prod_{r=1}^{s} \frac{(\partial_{{\ell}_r}P)^{|{k}_{r}|}}{{k}_r!({\ell}_r!)^{|{k}_r|}}\right]
,
\end{equation} where 
\begin{multline*}
p_s(\aa,{\lambda})=\{
({k}_1,\ldots,{k}_s;{\ell}_1,\ldots,{\ell}_s)\in(\N^{n})^{s}\times(\N^{d})^{s}:\, |{k}_i|>0,\\
 \00<{\ell}_1<\cdots<{\ell}_{s},\quad \sum_{i=1}^{s} {k}_i={\lambda},\quad\sum_{i=1}^{s}|{k}_i|{\ell}_i=\aa\}.
\end{multline*}

Note there are $\binom{n}{j}$ $n$-tuples 
${\lambda}\in\{0,1\}^{n}$ such that $|{\lambda}|=j$ and each of them is obtained from $\ee_1+\ee_2+\cdots+\ee_j$ by permuting the corresponding variables. Fixing one such ${\lambda}$, if $({k}_1,\ldots,{k}_s;{\ell}_1,\ldots,{\ell}_s)\in p_s(\aa,{\lambda})$, we see that $s\leq \sum_{i=1}^s |{k}_i|=|{\lambda}|=j$. Moreover, the term in brackets in (\ref{e169}) is independent of ${\lambda}$. Indeed, the previous permutation gives a bijective correspondence between $p_{s,j}(\aa):=p_s(\aa,\ee_1+\ee_2+\cdots+\ee_j)$ and $p_s(\aa,{\lambda})$. Therefore, (\ref{e169}) simplifies to \begin{equation}\label{e320}
\partial_{\aa}(P^n)= \sum_{j=1}^{|\aa|} \binom{n}{j} P^{n-j} \left[\sum_{s=1}^{j}\sum_{p_{s,j}(\aa)} \aa! \prod_{r=1}^{s} \frac{(\partial_{{\ell}_r}P)^{|{k}_{r}|}}{{k}_r!({\ell}_r!)^{|{k}_r|}}\right],
\end{equation} giving explicit formulas for $A_{\aa,j}$.
\end{nota}

\begin{lemma}\label{Lema412} Fix $m\geq1$ and $h,  P\in\C[[\xx]]$. Consider the differential operators $L_j$ in (\ref{L1general}) and the associated functions $L^{\star}_j(P)$ in (\ref{e157}). If $P$ divides $L_j^\star(P)$, for all $j=1,\dots,m$, then $P^m$ divides $\sum_{j=1}^{m} P^{j-1} L_j(hP^m)$.
\end{lemma}

\begin{proof} Using the multivariate Leibniz rule, \begin{equation}\label{Eq. MLR}
\d_{\aa}(hP^m)=\sum_{0\leq\bb\leq \aa} \binom{\aa}{\bb} \d_{\aa-\bb}(h)  \d_\bb(P^m),
\end{equation} we see that $$P^{j-1}L_j(hP^m)= P^{j-1}\sum_{|\aa|=j} a_{\aa}^{(j)} \sum_{0\leq\bb\leq \aa} \binom{\aa}{\bb} \d_{\aa-\bb}(h)  \d_\bb(P^m).$$ By Lemma \ref{lema178} we can write $$\partial_{\bb}(P^m)=\sum_{l=1}^{|\bb|}\frac{m!}{(m-l)!}P^{m-l}\cdot\AA_{\bb,l},$$ since $|\bb|\leq |\aa|=j\leq m$ and $\AA_{\bb,l}=0$ for $l>|\bb|$. To prove the statement, we analyze each one of the terms 
\begin{equation}\label{e283}
a_{\aa}^{(j)} \binom{\aa}{\bb} \d_{\aa-\bb}(h) \frac{m!}{(m-l)!} A_{\bb,l} P^{j-1+m-l},
\end{equation}
whose sum gives $P^{j-1} L_j(hP^m)$. We distinguish two  cases:
\begin{itemize}
\item If $|\bb|\leq j-1$, then $P^m$ divides (\ref{e283}) since $j-1+m-l\geq j-1+m-|\bb|\geq m$. \item If $|\bb|=j$, it holds that $\bb=\aa$. On the one hand, if $l<j$, then $j-1+m-l>m-1$ and we are done. Otherwise, $l=j$ and we are left with the term $$P^{j-1} \sum_{|\aa|=j} a_{\aa}^{(j)} h  A_{\aa,|\aa|} P^{m-j}=h L_j^\star(P)P^{m-1},$$ which, by hypothesis, is also divisible by $P^m$.
\end{itemize}
\end{proof}

The final lemma computes the derivatives of functions after a change of variables having $P$ as a holomorphic local coordinate.

\begin{lemma}\label{Lema 4} Let $P\in\C\{\xx\}$ such that $P(0)=0$ and $\d_{x_1}P(0)\neq 0$, and consider the change of variables $\xi:(\C^d,0)\to  (\C^d,0)$ given by \begin{equation}\label{Eq. xi}
\xi_1=P(\xx),\quad \xi_j=x_j,\quad j=2,\dots,d.
\end{equation}  If $f(\xi(\xx))=f(P(x),x_2,\dots,x_d)$ is holomorphic and $\aa\in\N^d\setminus\{0\}$, then $$\d_{\aa,\xx}(f)=\d_{\aa}^\star(P)\d_{\xi_1}^{|\aa|}(f)+\sum_{j=1}^{|\aa|-1} \left[ \sum_{\ast} B_{j,\beta}^{\aa}\cdot  \d_{\xi_1}^j\d_{\bb,\xi}(f)\right]+ \overline{\delta}_{\aa}\d_{\aa,\xi}(f),$$ where the inner sum is taken over all $\beta\in\N^{d-1}$ such that $(0,\beta)\leq \aa$ and $|\beta|\leq |\aa|-j$. The $B_{j,\beta}^{\aa}$ are polynomials in the derivatives of $P$ and  $\overline{\delta}_{\aa}:=(\overline{\delta}_{1,1})^{\a_1}\cdots (\overline{\delta}_{1,d})^{\a_d}$, where $\overline{\delta}_{i,j}:=1-\delta_{i,j}$ and $0^0=1$.
\end{lemma}

\begin{proof} Note that $\xi$ is indeed a holomorphic change of variables since $\xi(0)=0$ and its Jacobian determinant is precisely $\d_{x_1}(P)(0)\neq 0$. To prove the lemma we proceed by induction on $|\aa|$. In the case $|\aa|=1$, the chain rule shows that \begin{equation}\label{Eq. Aux Lema Convergent}
\d_{x_l}(f)=\d_{x_l}(P)\d_{\xi_1}(f)+\overline{\delta}_{1,l} \d_{\xi_l}(f),\quad l=1,\dots,d, 
\end{equation} proving this case. If we assume the result is valid up to some $|\aa|$, taking $l=1,\dots,d$, using the induction hypothesis and formula (\ref{Eq. Aux Lema Convergent}) we find that  \begin{align*}
&\d_{\aa+e_l,x}(f) =\d_{x_l}(\d_{\aa,x}(f))=\\
&\d_{x_l}(\d_{\aa}^\star(P)\d_{\xi_1}^{|\aa|}(f))+\sum_{j=1}^{|\aa|-1} \left[ \sum_{{(0,\beta)\leq \aa}\atop {|\beta|\leq |\aa|-j}} \d_{x_l}( B_{j,\beta}^{\aa}\cdot  \d_{\xi_1}^j\d_{\bb,\xi}(f) )\right] + \overline{\delta}_{\aa} \d_{x_l}(\d_{\aa,\xi}(f)),
\end{align*} which is equal to \begin{align}
\label{Eq. proof Lema3}
&\d_{x_l}(P)\d_{\aa,x}^\star(P) \d_{\xi_1}^{|\aa|+1}(f)+\overline{\delta}_{1,l}\d_{\aa}^\star(P)\d_{\xi_l}\d_{\xi_1}^{|\aa|}(f) +\d_{x_l}(\d_{\aa}^\star(P)) \d_{\xi_1}^{|\aa|}(f)+\\ \nonumber 
&\sum_{j=1}^{|\aa|-1} \!\! \sum_{{(0,\beta)\leq \aa}\atop {|\beta|\leq |\aa|-j}}  \!\! \d_{x_l}( B_{j,\beta}^{\aa}) \d_{\xi_1}^j\d_{\bb,\xi}(f) + B_{j,\beta}^{\aa}\left( \d_{x_l}(P)\d_{\xi_1}^{j+1}\d_{\beta,\xi}(f)+ \overline{\delta}_{1,l}\d_{\xi_1}^{j}\d_{\beta+e_l,\xi}(f)\right)\\ \nonumber
&+\overline{\delta}_{\aa} \d_{x_l}(P)\d_{\xi_1}\d_{\aa,\xi}(f)+\overline{\delta}_{\aa}\overline{\delta}_{1,l}\d_{\xi_l}\d_{\aa,\xi}(f).  
\end{align}

Note that the external terms are  $\d_{\aa+e_l,x}^\star(P) \d_{\xi_1}^{|\aa|+1}(f)$ and $\overline{\delta}_{\aa+e_l}\d_{\aa+e_l,\xi}(f)$ as required. On the other hand, the remaining terms have the form $B_{k,\gamma}^{\aa+e_l}\d_{\xi_1}^k \d_{\gamma,\xi}(f)$ with $1\leq k\leq |\aa|$ and $(0,\gamma)\leq \aa+e_l$, where the $B_{k,\gamma}^{\aa+e_l}$ can be found recursively. By the nature of the terms in the sum (\ref{Eq. proof Lema3}) it is clear that each $B_{k,\gamma}^{\aa+e_l}$ is a polynomial in the derivatives of $P$. The principle of induction allows to conclude the proof.
\end{proof}

\begin{nota} Another way to prove Lemma \ref{Lema 4} is applying Fa\`{a} di Bruno formula \cite[p. 505]{cosa} to $f(P(\xx),x_2,\dots,x_d)$. On the other hand, for our purposes it is not necessary to specify the recurrences to determine the coefficients $B^{\aa}_{j,\beta}$. However, for $\aa=n e_1$ and $l=1$, we have that $\beta=0$ for all $j$ and (\ref{Eq. proof Lema3}) takes the form \begin{align*}
\d_{x_1}^{n+1}(f)=&\d_{x_1}(P)^{n+1}\d_{\xi_1}^{n+1}(f)+\d_{x_1}((\d_{x_1}(P)^n))\d_{x_1}^n(f)\\
&+\sum_{j=1}^{n-1} \left[\d_{x_1}(B_{j,0}^{ne_1}) \d_{\xi_1}^j(f)+B_{j,0}^{ne_1} \d_{x_1}(P) \d_{\xi_1}^{j+1}(f)\right].
\end{align*} Setting $B_{n,0}^{ne_1}=(\d_{x_1}P)^n$, we find that $$B_{1,0}^{(n+1e_1)}=\d_{x_1}(B_{1,0}^{ne_1}),\quad B_{j,0}^{(n+1)e_1}=\d_{x_1}(B_{j,0}^{ne_1})+\d_{x_1}(P)B_{j-1,0}^{ne_1},\qquad j=2,\dots,n.$$ This is recurrence (\ref{Eq. 21}) in Lemma \ref{lema178}. Thus $B_{1,0}^{ne_1}=A_{ne_1,1}=\d_{x_1}^n(P)$,   $B_{j,0}^{ne_1}=A_{ne_1,j}$, $j=2,\dots,n-1$, and $$\d_{x_1}^n(f)=(\d_{x_1}P)^n \d_{\xi_1}^n (f)+\sum_{j=1}^n A_{ne_1,j} \d_{\xi_1}^j(f),\qquad \text{ for all } n\geq 1.$$

\end{nota}

\section{The proof of Theorems \ref{teopral} and \ref{Thm. Convergent}}\label{Sec. The proof of the main results}

The idea behind the proofs is simple. For Theorem \ref{teopral} we add a variable $t$ and working in $(t,x)\in(\C\times\C^{d},0)$ we search for a PDE satisfied by the series $\widehat{w}=\sum y_n t^n$, where $\widehat{y}=\sum y_n P^n$ is the solution to the initial equation (\ref{e1general}). For Theorem \ref{Thm. Convergent}, we can take $P=\xi_1$ as one of the variables, and write (\ref{e1general}) as an equation in the new coordinates. In both cases the existence, uniqueness and Gevrey type will be obtained from Theorem \ref{Thm. Gerard-Tahara1}.

\begin{proof}[Proof of Theorem \ref{teopral}]

We point out that if $F(x,0)\equiv 0$, then the the unique formal power solution is zero. Thus we assume $f(x):=F(x,0)\not\equiv 0$. We will write 
\begin{equation}\label{eq:F}
F(\xx,\yy)=f(\xx)+A(\xx)\yy+H(\xx,\yy),\quad H(\xx,\yy)=\sum_{I\in\N^N, |I|\geq 2} A_I(\xx) \yy^I,
\end{equation}
where $f\in\mathcal{O}_b(D_r^d,\C^N)$ with $f(\00)=\00$, $A\in\mathcal{O}_b(D_r^d,\C^{N\times N})$,  and $H\in\mathcal{O}_b(D_r^d\times D_r^N,\C^{N})$ has no constant nor linear terms in its Taylor expansion with respect to $\yy$ at the origin, and where $r>0$ is small. Since $A(\00)=D_yF(0,0)$ is invertible, by continuity we can assume $A(\xx)$ is also invertible for all $x\in D_r^d$. 

We search for a formal $P$-series solution of (\ref{e1general}) in the form
\begin{equation}\label{e53b}
\widehat{\yy}(\xx)=\sum_{n=0}^\infty \yy_n(\xx)P(\xx)^n,
\end{equation}
with the $\yy_n(\xx)\in \mathcal{O}_b(D_\rho^d,\C^N)$, for all $n\ge0$, for a common $\rho>0$. The rest of the proof is divided in several steps.

\textbf{Step 1}: We determine the terms $\yy_0(\xx),\dots,\yy_{k-1}(\xx)$ inductively solving adequate implicit equations. For the coefficient $\yy_0$, setting $x=0$ in (\ref{e1general}) and recalling that $F(0,0)=0$, we require that $\yy_0(\00)=\00$. Now we search for a holomorphic solution of   \begin{equation}\label{e67}
f(\xx)+A(\xx)\yy_0(\xx)+H(\xx,\yy_0(\xx))=0.
\end{equation} Since $A(0)$ is invertible, shrinking $r>0$ if necessary, the implicit function theorem leads to the existence of such solution  $\yy_0(\xx)\in\mathcal{O}_b(D_r^d,\C^N)$ with $\yy_0(\00)=\00$. Then, considering the change of variables $\yy=\yy_0+\ww_0$ in (\ref{e1general}), we find that $w_0$ satisfies the system  \begin{equation}\label{e212}
\sum_{j=1}^k P^{j}L_j(\ww_0)
=F_0(\xx,\ww_0)=g_0(\xx)+B_0(\xx)\ww_0+H_0(\xx,\ww_0),
\end{equation} with $F_0(\00,\00)=\00$, the matrix $B_0(\00)$ is invertible, and the Taylor expansion of $H_0\in\mathcal{O}_b(D_\rho^d\times D_\rho^N,\C^N)$ with respect to $\ww_0$ has no constant nor linear terms. Here, we have written 
$$g_0:=-\sum_{j=1}^k P^{j}L_j(\yy_0),\qquad B_0:=A+A_0,$$ where $A_0\in\mathcal{O}_b(D_\rho^d,\C^{N\times N})$ satisfies $A_0(\00)=\00$. These maps are obtained from  $$H(\xx,\yy_0(\xx)+\ww_0)-H(\xx,\yy_0(\xx))=A_0(\xx)\ww_0+H_0(\xx,\ww_0),
$$  $$A_0(\xx)=\left.D_{w_0}( H(\xx,\yy_0(\xx)+\ww_0))\right|_{w_0=0}=D_\yy H(\xx,\yy_0(\xx)).$$ 

We now proceed recursively, by means of the change of variables $$\ww_{m-1}=\ww_{m}+\yy_{m} P^{m},\qquad m=1,\dots,k-1,$$ and determining functions $g_m$, $A_m, B_m$ and $H_m$ defined by \begin{align}
\label{eq.271}& A_{m}\ww_{m}+H_{m}(\xx,\ww_{m})=H_{m-1}(\xx,\ww_{m}+\yy_{m} P^{m})-H_{m-1}(\xx,\yy_{m} P^{m})\\
\nonumber & g_m=-\sum_{j=1}^k P^{j}L_j(\yy_m P^m),\qquad B_m=B_{m-1}+A_m=A+(A_0+\cdots+A_m).
\end{align}  Note that $H_m$ has no constant or linear terms in its Taylor expansion in $\ww_m$ near the origin and that $B_m(\00)=A(\00)$ is an invertible matrix. Indeed, we see from  (\ref{eq.271}) that $A_{m}=\left. D_{ \ww_{m}}(H_{m-1}(\xx,\ww_{m}+\yy_{m}P^{m}))\right|_{w_{m}=0}=D_{\ww_{m-1}} H_{m-1}(\xx,\yy_{m} P^{m}),$ so $A_m(\00)=\00$ as required.

To proceed we need to define $\yy_m$ in a consistent way. If  $\yy_{m-1}$, $g_{m-1}$, $B_{m-1}$ and $H_{m-1}$ have been found, we set $y_m$ as the unique holomorphic solution near the origin of the system 
$$P^{-m}g_{m-1}+B_{m-1}\yy_{m}+P^{-m}H_{m-1}(\xx,\yy_{m} P^{m})=\00.$$ 
This equation has holomorphic coefficients on some neighborhood of the origin. Indeed, the function $g_{m-1}$ is divisible by $P^{m}$ thanks to Lemma \ref{Lema412} since $g_{m-1}=-P\cdot\sum_{j=1}^{m-1} P^{j-1}L_j(y_{m-1}P^{m-1})-\sum_{j=m}^k P^{j}L_j(y_{m-1}P^{m-1})$. Also, if we write $$H_m(\xx,\ww_m)=\sum_{|I|\geq 2} A_{I,m}(\xx) \ww_m^I,$$ then 
$$P^{-m}H_{m-1}(\xx,\yy_{m} P^{m})=\sum_{|I|\geq 2} A_{I,m-1}(\xx) P^{m(|I|-1)} \yy_m^I,$$ which also has holomorphic coefficients in $\xx$ that vanish at $\xx=\00$. Therefore, $\yy_m$ is determined by means of the implicit function theorem.

At this point it follows from a direct recursive argument that $w_m$ satisfies 
\begin{equation}\label{e250a}
\sum_{j=1}^k P^{j}L_j(\ww_{m})
=F_m(\xx,\ww_m):=g_{m}(\xx)+B_{m}(\xx)\ww_{m}+H_{m}(\xx,\ww_{m}),
\end{equation}  where $F_m(\00,\00)=0$, $B_m(\00)$ is invertible, and $H_m(\xx,\ww_m)$ has no constant nor linear terms in its Taylor expansion in $\ww_m$ in a neighborhood of the origin. In conclusion, after collecting all the previous changes of variables we find that $\ww=\ww_k$ defined by  
$\ww=\yy-(y_0+y_1P+\cdots+y_{k-1} P^{k-1})$ satisfies \begin{equation}\label{e250}
\sum_{j=1}^k P^{j}L_j(\ww)
=g(\xx)+B(\xx)\ww+H'(\xx,\ww),
\end{equation} where $g:=g_{k-1}\in\mathcal{O}_b(D_r^d,\C^N)$ is divisible by $P^k$, $B:=B_{k-1}\in\mathcal{O}_b(D_r^d,\C^{N\times N})$ with $B(0)$ invertible, and $H'=H_{k-1}\in\mathcal{O}_b(D_r^d\times D_r^N,\C^{N})$ has no constant nor linear terms in its Taylor expansion with respect to $\ww$, and where $r>0$ has been reduced when required. Therefore, we can restrict the problem to find a solution of (\ref{e250}) having the form $$\widehat{\ww}(\xx)=\sum_{n=k}^\infty y_n(\xx) P(\xx)^n.$$

\textbf{Step 2}: We study the action of the operator $P^jL_{j}$ on $\widehat{\ww}(\xx)$ for each $j=1,\dots,k$. By the multivariate Leibniz rule (\ref{Eq. MLR}) and Lemma \ref{lema178} we see that \begin{align*}
L_j(\ww)=&\sum_{|\aa|=j} a_{\aa}^{(j)}  \sum_{n=k}^\infty\d_{\aa}(y_nP^n)\\
=&\sum_{n=k}^\infty L_j(\yy_n) P^n + \sum_{n=k}^\infty \sum_{|\aa|=j} a_{\aa}^{(j)} \sum_{0<\bb\leq \aa} \binom{\aa}{\bb}  \partial_{\aa-\bb} (y_n) \sum_{l=1}^{|\bb|} \frac{n!}{(n-l)!} P^{n-l} \AA_{\bb,l},\\
=&\sum_{n=k}^\infty L_j(\yy_n) P^n +S_j,
\end{align*} where the first sum corresponds to $\bb=\00$. Note that the last inner sum is taken over $1\leq l \leq |\bb|$ since $A_{\bb,l}=0$ if $l>|\bb|$, and $|\bb|\leq |\aa|=j\leq k\leq n$. Let us write $S_j=S_{j,1}+S_{j,2}$, where $S_{j,2}$ retains the terms corresponding to $\bb=\aa$. Then $$S_{j,2}=S_{j,3}+\sum_{n=k}^\infty  \yy_n \frac{n!}{(n-j)!} P^{n-j} L_j^\star(P),$$ where $S_{j,3}$ contains the terms in which $l<j$ and $S_{j,2}-S_{j,3}$ in the previous expression corresponds to $l=j$, according to the definition of $L_j^\star$. Therefore, $$
S_{j,3} = \sum_{n=k}^\infty  \sum_{l=1}^{j-1} \left[\sum_{|\aa|=j} a_{\aa}^{(j)} \AA_{\aa,l}\right] \yy_n \frac{n!}{(n-l)!} P^{n-l}.$$
 On the other hand, we can organize the terms in $S_{j,1}$ to write
\begin{align*}
S_{j,1} &= \sum_{n=k}^\infty  \sum_{m=1}^{j-1} \sum_{l=1}^{m} \left[ \sum_{{|\aa|=j, }\atop {|\bb|=m, \bb<\aa} }  \binom{\aa}{\bb}  a_{\aa}^{(j)}  \partial_{\aa-\bb} (y_n) \AA_{\bb,l} \right]  \frac{n!}{(n-l)!} P^{n-l} \\
&= \sum_{n=k}^\infty  \sum_{l=1}^{j-1}  \left[ \sum_{m=l}^{j-1} \sum_{{|\aa|=j, }\atop {|\bb|=m, \bb<\aa} }  \binom{\aa}{\bb}  a_{\aa}^{(j)}  \partial_{\aa-\bb} (y_n) \AA_{\bb,l}\right]  \frac{n!}{(n-l)!} P^{n-l},
\end{align*} by grouping those indices $\bb$ with the same norm.

\

\textbf{Step 3}: We search for a partial differential equation satisfied by \begin{equation}\label{e307}
\widehat{W}(t,\xx):=\sum_{n=k}^\infty \yy_n(\xx) t^n,
\end{equation} from the system (\ref{e250}) satisfied by $\widehat{\ww}(\xx)=\widehat{W}(P(\xx),\xx)$. Indeed, recalling (\ref{e138}) we can write $L^{\star}_{j}(P)=\phi_j\cdot P$, for some holomorphic function $\phi_j$ near the origin. Therefore, noticing that $\frac{n!}{(n-l)!} P^{n-l}=\d_t^l(t^n)|_{t=P}$, for $l\leq n$, we find $$P^j\sum_{n=k}^\infty  \yy_n \frac{n!}{(n-j)!} P^{n-j} L_j^\star(P)=\left.\phi_j t^{j+1}\d_t^j(\widehat{W})\right|_{t=P}.$$ Let us  consider the differential operator \begin{align*}
K(\xx)&(t,D^kW):=-H'(\xx,W)+\sum_{j=1}^k \Bigg[ t^j L_j(W)+\phi_j t^{j+1}\d_t^j(W)\\
&+\sum_{l=1}^{j-1} \Bigg(\sum_{|\aa|=j} a_{\aa}^{(j)} \AA_{\aa,l}  t^j\d_t^l(W)+   \sum_{m=l}^{j-1} \sum_{{|\aa|=j, }\atop {|\bb|=m, \bb<\aa} }  \binom{\aa}{\bb}  a_{\aa}^{(j)} \AA_{\bb,l} \partial_{\aa-\bb}t^j\d_t^l(W)\Bigg)\Bigg].
\end{align*} Then $\widehat{\ww}(x)=\sum_{n=k}^\infty y_n P^n$ satisfies (\ref{e250}) if and only if $\widehat{W}$ in (\ref{e307}) satisfies $$B(x)W=-h(\xx)t^k+K(x)(t,D^kW),$$ where 
$g=h \cdot P^k$. Theorem \ref{Thm. Gerard-Tahara1} for  $p=0$ proves that this equation has a unique formal power series solution $\widehat{W}(t,\xx)$ of the form (\ref{e307}). Therefore, we have the existence and uniqueness of the solution $w$ of equation (\ref{e250}) and therefore of the main equation (\ref{e1general}). 

Finally, Theorem \ref{Thm. Gerard-Tahara1} also asserts that  $\widehat{W}(t,\xx)$ is $s$-Gevrey in $t$ where $s$ in (\ref{Eq. Gevrey class}) is computed using the  derivatives appearing in $K$. In this case, $s_0(t^j \d_{\aa})=|\aa|/j=1$, for the terms in $t^j L_j$,  \begin{align*}
s_0(t^{j+1}\d_t^j)&=j,\quad s_0(t^{j}\d_t^l)=\frac{l}{j-l},\text{ and }\quad s_0(t^{j}\d_t^l \d_{\aa-\bb})=\frac{l+|\aa|-|\bb|}{j-l},
\end{align*} where $1\leq j\leq k$, $l\leq j-1$, $|\aa|=j$, $0<\bb\leq \aa$, and $l\leq |\bb|\leq j-1$. Thus $$s_0(t^{j}\d_t^l)\leq \frac{j-1}{j-l}\leq j-1,\quad \text{ and }\quad s_0(t^{j}\d_t^l \d_{\aa-\bb})\leq \frac{|\aa|}{j-l}\leq j.$$ Therefore, the maximum $s$ of these values is $k$ and it attained at the term $\phi_k t^{k+1}\d_t^k$, when $\phi_k\neq 0$. If $\phi_k=0$, we still have that $s=k$ as it is also attained at the terms $a^{(k)}_{\aa} A_{\bb,|\bb|}\d_{\aa-\bb}  t^k\d_t^{k-1}$, where $|\aa|=k$ and $\bb<\aa$ with $|\bb|=k-1$. But $L_k\neq 0$, so there is $\aa_0\in\N^d$ with $|\aa_0|=k$ and  $a_{\aa_0}^{(k)}\neq 0$. Recalling formula (\ref{Eq. DaP^n}) we see that at least one of these terms appears in $K$, thus $s=k$. In conclusion, $\widehat{W}$ is a $k$-Gevrey series in $t$, i.e, $\widehat{w}$ is a $P$-$k$-Gevrey series as we wanted to show.
\end{proof}

It is worth remarking that the proof of Theorem \ref{teopral} simplifies considerably when $k=1$. To highlight the main ideas used we reproduce the argument again.

\begin{coro}\label{teo2}  
Consider the partial differential equation \begin{equation}\label{e1}
P(\xx)L_1(\yy)
=F(\xx,\yy),
\end{equation} where $L_1=a_1\d_{x_1}+\cdots+a_d \d_{x_d}$ has holomorphic coefficients at the origin, $F$ is holomorphic near the origin, $F(\00,\00)=\00$, and $D_yF(0,0)$ is an invertible matrix. If $P$ divides $L_1(P)$,  equation (\ref{e1}) has a unique formal power series solution $\widehat{\yy}\in\C[[x]]^{N}$ with $\widehat{y}(0)=0$, which is a $P$-$1$-Gevrey series.
\end{coro}

\begin{proof} Writing $F$ as in equation (\ref{eq:F}) and setting $y=y_0+w$ in the equation (\ref{e1}), where $y_0$ solves the implicit equation (\ref{e67}), we find that $w$ satisfies \begin{equation}\label{e409}
P\cdot L_1(\ww)
=g_0(\xx)+B_0(\xx)\ww+H_0(\xx,\ww),
\end{equation} where $g_0=-P\cdot L_1(y_0)$, $B_0(x)$ is invertible at $x=0$, and the Taylor expansion of $H_0(x,w)$ in $w$ has no constant nor linear terms. This reduces the problem to find a formal solution  $\widehat{\ww}(\xx)=\widehat{W}(x,P(x)),$ where $\widehat{W}(x,t)=\sum_{n=1}^\infty y_n(x)t^n$. Note that \begin{align*}
P\cdot L_1(\widehat{\ww})
=\sum_{n=1}^\infty L_1(\yy_n)P^{n+1}+ \phi \cdot n\yy_n   P^{n+1} = \left. \left(tL_1+\phi t^2\partial_t\right)(\widehat{W})\right|_{t=P},
\end{align*} where $L_1(P)=\phi\cdot P$. Therefore, $\widehat{w}$ solves  (\ref{e409}) if and only $\widehat{W}$ solves 
\begin{equation}\label{e15}
B_0(\xx)W=-L_1(y_0)t+\left(tL_1+\phi(\xx) t^2\partial_t\right)W-H_0(\xx,W).
\end{equation}

Theorem \ref{Thm. Gerard-Tahara1} for $p=0$ shows that (\ref{e15}) has a unique formal power series solution $\widehat{W}(x,t)$ where the $y_n$ are holomorphic functions in a common neighborhood of $\00\in\C^d$. Moreover, $\widehat{W}$ is $s$-Gevrey, where $s$ is the of $s_0(\phi t^2\partial_t)=1$ and $s_0(t\partial_{x_j})=1$, for $j=1,\dots,d$ such that $a_j\neq 0$. Since $L_1\not\equiv0$, it follows that $s=1$ as required.
\end{proof}

We move now to Theorem \ref{Thm. Convergent}. Although we can apply the same technique as in Theorem \ref{teopral}, it is easier to directly take $P$ as one of the coordinates.

\begin{proof}[Proof of Theorem \ref{Thm. Convergent}] By hypothesis $L_k^{\ast}(P)(0)=\sum_{|\aa|=k} a_{\aa}^{(k)}(0) \d_{\aa}^\star(P)(0)\neq 0$. Thus at least one of these terms is non-zero, so necessarily $\d_{x_l}(P)(0)\neq 0$ for some $l=1,\dots,d$ ---recall (\ref{e337})---. Up to permuting the coordinates we can assume that $l=1$. We make the change of variables (\ref{Eq. xi}) and write equation (\ref{e1general}) in the coordinates $\xi=(\xi_1,\xi')$, $\xi':=(\xi_2,\dots,\xi_d)$.  In fact, setting $u(\xi)=y(\xx)$, Lemma \ref{Lema 4} shows that $$\sum_{j=1}^k P^j L_j(y)=\sum_{j=1}^k L_j^\star(P) \xi_1^j \d_{\xi_1}^{j}(u)+\xi_1^j C_j(\xi,u,D^ju),$$ $$C_j(\xi,u,D^ju):=\sum_{|\aa|=j} \overline{a}_{\aa}^{j}(\xi)\left[\overline{\delta}_{\aa}\d_{\aa,\xi}(u)+\sum_{l=1}^{j-1}  \sum_{\ast} B_{l,\beta}^{\aa}\cdot  \d_{\xi_1}^l\d_{\bb,\xi}(u)\right],$$ where the inner sum is taken over all $\beta\in\N^{d-1}$ such that $(0,\beta)\leq \aa$ and $|\beta|\leq |\aa|-l$, and $\overline{a}_{\aa}^{j}(\xi)=a_{\aa}^{(j)}(\xx)$. Therefore, $y(x)=\sum_{\bb\in \N^d} y_{\bb} \xx^{\beta}\in\C[[x]]^N$ is a solution of (\ref{e1general}) if and only if $u(\xi)=y(x(\xi))=\sum_{n=0}^\infty u_n(\xi')\xi_1^n\in\C[[\xi]]^N$ satisfies \begin{equation}\label{Eq. Aux Thm2}
\sum_{j=1}^k L_j^\star(P) \xi_1^j \d_{\xi_1}^{j}(u)=\overline{F}(\xi,u)-\sum_{j=1}^k \xi_1^j C_j(\xi,u,D^ju),
\end{equation} where $\overline{F}(\xi,u)=F(\xx,\yy)$.  Write $F$ as in equation (\ref{eq:F}), and expand $A(x)=\overline{A}(\xi)=A_0(\xi')+\sum_{m= 1}^\infty  A_m(\xi')\xi_1^m$ in powers of $\xi_1$, where $A_0(0)=A_0$. Then we conclude that (\ref{Eq. Aux Thm2}) has the form of equation (\ref{Eq. Main GT}) with $p=k$, and $$c_0(\xi')=A_0(\xi'),\qquad c_j=L_j^\star(P)I_N,\qquad j=1,\dots,k.$$ Now, Remark \ref{Rrm Stirling} and the hypothesis (\ref{Eq. Condition Thm2}) guarantee that we can apply Theorem \ref{Thm. Gerard-Tahara1} to (\ref{Eq. Aux Thm2}) to conclude the existence and uniqueness of the solution ${u}(\xi)\in \C[[\xi]]^N$ which is $s$-Gevrey, with $s$ as in (\ref{Eq. s}). The terms that appear in  (\ref{Eq. Aux Thm2}) satisfy $$s_k(\xi_1^j\d_{\aa,\xi})= \max\left\{0,\frac{|\aa|-k}{j}\right\}=0,\quad s_k(\xi_1^j\d_{\xi_1}^l\d_{\bb,\xi})=\max\left\{0,\frac{|\bb|+l-k}{j-l}\right\}=0,$$ because $|\bb|+l-k\leq |\aa|-k=j-k\leq 0$. Therefore, $s=0$ and ${u}(\xi)={y}(x)\in\C\{x\}^N$ is convergent as we wanted to show.
\end{proof}

For the case $k=1$ Corollary \ref{teo2} takes the following form, c.f. \cite[Theorem 2]{cahu}.

\begin{coro}\label{Coro 2} Assume the conditions of Corollary \ref{teo2}, but now suppose that $L_1(P)(\00)\neq 0$. If $nL_1(P)(\00)I_N-D_yF(0,0)\in\textup{GL}_N(\C)$, for all $n\in\N$, equation (\ref{e1}) has a unique analytic solution at the origin $\widehat{\yy}\in \C\{\xx\}^N$ with $\widehat{y}(0)=0$.
\end{coro}

\section{Examples}\label{Sec. Examples}

We include some worked examples. In particular, Example \ref{eje3} shows that the Gevrey type provided by Theorem \ref{teopral} is attained, thus, in general it cannot be improved. For more examples in the case $k=1$ we refer to \cite{cahu}, including the use of  ramifications and punctual blow-ups (\ref{Eq. Blow up}) to bring other differential equations into a form where Theorem \ref{teopral} can be applied.

\begin{eje}\label{eje1} Fix integers $m,k\geq 1$ and consider the scalar equation \begin{equation}\label{Eq. Ex1}
x^{(m+1)k}\d_{x}^{k}y=y-1-\frac{x^k}{k!}.
\end{equation} It has a unique formal solution  
$\widehat{y}(x)=\sum_{n=0}^{\infty}{y_n}x^n$ given by $$y_0=y_k=y_{mk+k}=1,\qquad y_{jmk+k}=\prod_{l=1}^{j-1} \frac{(lmk+k)!}{(lmk)!},\quad j\geq 2,$$ and $y_n=0$ in other cases. But $(lmk+k)!/(lmk)!=(lmk+1)\cdots(lmk+k)\leq (lmk+k)^k\leq ((j-1)mk+k)^k=k^k((j-1)m+1)^k\leq (2km)^k(j-1)^k$. Thus $$a_j:={y_{jmk+k}}\leq (2km)^{k(j-1)}(j-1)^{k(j-1)},$$ and \begin{equation}\label{Eq. Ex1.1}
\widehat{y}(x)=1+x^k\cdot\sum_{j=0}^\infty {a_{j}} (x^{mk})^j\,\, \text{ is  $x^{mk}$-$k$-Gevrey, i.e., it is $x$-$1/m$-Gevrey}.
\end{equation} On the other hand, for $k\geq 2$ we can apply Theorem \ref{teopral} to $P(x)=x^{m+1}$ and $L_k=\d_x^k$ to conclude that $\widehat{y}$ is $x^{m+1}$-$k$-Gevrey, i.e., $x$-$\frac{k}{m+1}$-Gevrey. In fact,  $$L_k^\star(P)=((m+1)x^m)^k \text{ is divisible by } P,$$ since $m+1\leq mk$. If $m+1<mk$, (\ref{Eq. Ex1.1}) gives a better bound. However, if $m+1=mk$, then $m=1$, $k=2$, and  Theorem \ref{teopral} gives an optimal bound. Indeed, $$\widehat{y}(x)=1+\frac{x^2}{2}\sum_{j=0}^\infty (2j)!x^{2j}$$ which is exactly $x^2$-$2$-Gevrey, i.e., $x$-$1$-Gevrey.
\end{eje}

\begin{eje}\label{eje3} Consider the scalar equation $$x_1^2x_2^2(x_1^2\d_{x_1}^2u+x_2^2\d_{x_2}^2u+2\d_{x_1x_2}u)-2u=2x_1x_2,$$ having as unique formal power series solution  $\widehat{u}(x_1,x_2)=\sum_{n=0}^\infty a_n x_1^n x_2^n$. In fact, this equation corresponds to the ODE $$t^4\d_t^4w+t(t\d_t)^2w-w=t,$$ where $t=x_1x_2$ and $u(x_1,x_2)=w(t)$. Theorem \ref{Thm. Gerard-Tahara1} proves that $\widehat{w}(t)=\widehat{u}(x_1,x_2)$ is $t$-$2$-Gevrey. Theorem \ref{teopral} applied to  $k=2$, $P=x_1x_2$, $L_1=0$, and $$ L_2=x_1^2\d_{x_1}^2+x_2^2\d_{x_2}^2+2\d_{x_1x_2},$$ shows $\widehat{u}$ is $x_1x_2$-2-Gevrey since $L_1^\star(P)=0$ and  $L_2^\star(P)=2P^2+2P=2P(1+P)$. 

We can also find the Gevrey order by direct means. First, the $a_n$ are given by $a_0=0$, $a_1=a_2=1$, and $$ a_n=(n-1)^2 a_{n-1}+(n-2)(n-3)a_{n-2},\qquad \text{ for } n\geq 3.$$ If we set $\a_n=a_n/(n-1)!^2$, $n\geq 1$, this sequence satisfies $$\a_n=\a_{n-1}+\frac{n-3}{(n-1)^2(n-2)}\a_{n-2}.$$ It follows by induction that $1\leq \a_n\leq \varphi^n$, where $\varphi=(1+\sqrt{5})/2$ solves  $\varphi^2=\varphi+1$. In conclusion, $(n-1)!^2\leq a_n\leq \varphi^n (n-1)!^2,$ so $\widehat{u}$ is exactly $x_1x_2$-$2$-Gevrey. Thus, the Gevrey type provided by Theorem \ref{teopral} cannot be improved.
\end{eje}

\begin{eje}\label{ejeLast} Returning to the framework of singular perturbations, we consider systems $$\epsilon^k x^{k+1}\d_x^ky+\sum_{j=1}^{k-1} \epsilon^j x^{j+1} a_j(x,\epsilon)\d_x^jy=F(x,\epsilon,y),$$ where  $\C\ni\epsilon\to 0$, $x\in\C$, the $a_j$ are holomorphic near $(0,0)\in\C^2$, and $y$ and $F$ are as in Theorem \ref{teopral}. The main result can be applied to $P(x,\epsilon)=x\epsilon$, $L_k=x\d_x^k$ and $L_j=a_j x\d_x^j$ since $$L_k^\star(P)=x(\d_xP)^k=x\epsilon^k,\quad \text{ and }\quad  L_j^\star(P)=a_j x(\d_xP)^j=a_jx\epsilon^j,$$ are divisible by $P$. Therefore, this system has a unique formal power series solution in $x$ and $\epsilon$ which is $x\epsilon$-$k$-Gevrey. The case $k=1$ was first established in \cite{CDRSS}.
\end{eje}

\begin{eje} Fix $\aa\in\N^d\setminus\{0\}$ and consider the equation $$(\xx^\aa)^k L_k(y)(x)+\cdots+\xx^\aa L_1(y)(x)=F(\xx,\yy),$$ with  differential operators of the form $L_j=\sum_{|\bb|=j} b_{\bb}(\xx) \xx^{\bb} \d_{\bb},$ where $b_{\bb}\in\mathcal{O}_b(D_r^d)$, for a common $r>0$. Assuming that $D_yF(0,0)$ is invertible, since $$L_j^\star(\xx^\aa)= \xx^{j\aa}\cdot \sum_{|\bb|=j} \aa^\bb b_{\bb}(x),$$ is divisible  by $\xx^\aa$ for all $j=1,\dots,d$, Theorem \ref{teopral} proves that this equation has a unique $\xx^{\aa}$-$k$-Gevrey series solution. Note that equation (\ref{Eq. Intro X}) is a particular case for $k=1$, where this Gevrey bound is optimal due to the $x^\a$-$1$-summability of the solution and Tauberian theorems for these methods, \cite{Carr1,CMS19}.
\end{eje}

\end{document}